\documentclass[12pt, leqno]{article}

\usepackage{latexsym}
\usepackage{indentfirst}
\usepackage{graphicx}
\usepackage{amsmath,amsfonts}
\usepackage{amssymb}
\usepackage{amsthm}
\usepackage{bm}
\usepackage{accents}
\usepackage{changepage}
\usepackage{color}
\usepackage{bigints}
\usepackage{enumerate}
\usepackage{amsmath, nccmath}
\usepackage[margin=1in]{geometry}
\usepackage{yhmath}
\usepackage{dirtytalk}
\usepackage{cleveref}
\usepackage{comment}

\newtheorem{prop}{Proposition}[section]
\newtheorem*{defi*}{Definition}

\newtheorem{exam}[prop]{Example}

\newtheorem{rem}[prop]{Remark}
\newtheorem{thm}[prop]{Theorem}
\newtheorem{coro}[prop]{Corollary}

\newcommand{\Pp}{\mathbb{P}} 
\newcommand{\cb}{{\mathcal B}} 
\newcommand{\cd}{{\mathcal D}} 
\newcommand{\ce}{{\mathcal E}} 
\newcommand{\cu}{{\mathcal U}} 
\newcommand{\R}{{\mathbb{R}}}

\numberwithin{equation}{section}

\title{\large Strong Feller semigroups and Markov processes: A counter example}

\author{
Lucian Beznea
\footnote{Simion Stoilow Institute of Mathematics  of the Romanian Academy,
 Research unit No. 2, 
P.O. Box \mbox{1-764,} RO-014700 Bucharest, Romania
(e-mail: lucian.beznea@imar.ro)}
\and
Iulian C\^{i}mpean
\footnote{Faculty of Mathematics and Computer Science, University of Bucharest, 14~Academiei, 010014~Bucharest,
and ``Simion Stoilow" Institute of Mathematics of the Romanian Academy, 21~Calea Grivi\c{t}ei, 010702~Bucharest, Romania
(e-mail: iulian.cimpean@unibuc.ro)}
\and
Michael R\"ockner
\footnote{Fakult\"at f\"ur Mathematik, Universit\"at Bielefeld,
Postfach 100 131, D-33501 Bielefeld, Germany, and Academy for Mathematics and Systems Science, CAS, Beijing 
(e-mail: roeckner@mathematik.uni-bielefeld.de)}
}
\date{}

\begin{document}

\maketitle
\begin{abstract}
The aim of this note is to show, by providing an elementary way to construct counter-examples, that the strong Feller and the joint (space-time) continuity  for a semigroup of Markov kernels on a Polish space are not enough to ensure the existence of an associated c\`adl\`ag Markov process on the same space. 
One such simple counter-example is the Brownian semigroup on $\mathbb{R}$ restricted to $\mathbb{R}\setminus \{0\}$, for which it is shown that there is no associated c\`adl\`ag Markov process.
Using the same idea and results from potential theory we then prove that the analogous result with c\`adl\`ag Markov process replaced by right Markov process also holds, even if one allows to change the Polish topology to another Polish topology with the same Borel $\sigma$-algebra.
\end{abstract}

%\vspace{2mm}

\noindent
{\bf Mathematics Subject Classification (2010):} 
%60H15,     % Stochastic partial differential equations [See also 35R60]
60J40,  	% Right processes
%60H10,     % Stochastic ordinary differential equations [See also 34F05]
60J45,  	% Probabilistic potential theory
60J35, 	% Transition functions, generators and resolvents
%60J57,  	% Multiplicative functionals
%31C25,      % Dirichlet spaces
47D07,  	% Markov semigroups and applications to diffusion processes 
%35R60,       % Partial differential equations with randomness, stochastic partial differential equations
60J25.        % Continuous-time Markov processes on general state spaces
%37C40 (primary), 37A30, 37L40, 60J35, 60J25, 31C25, 37C40,  82B10 (secondary)
%47D03  	%Groups and semigroups of linear operators 
%47A35  	%Ergodic theory
%37A30   %Ergodic theorems, spectral theory, Markov operators
%37C40   %Smooth ergodic theory, invariant measures
%37L40   %Invariant measures (Infinite-dimensional dissipative dynamical systems)
%60J55  	Local time and additive functionals
%60J25  	Continuous-time Markov processes on general state spaces
%82B10  	Quantum equilibrium statistical mechanics (general)
%31C05  	Harmonic, subharmonic, superharmonic functions

\section{Introduction}
It seems that there is an often asked  question
%there was a kind of  ''folklore belief'' 
in the field whether for a semigroup of Markov kernels $(P_t)_{t \geq 0}$ on a Polish space $E$, which has the strong Feller property, i.e.
\begin{align}\tag{1.1}\label{eq:1.1}
f \in \mathcal{B}_b(E) \Longrightarrow P_t f:= \int_E f(y) P_t(\cdot , dy) \in C_b(E)\ \mbox{ for all } t \geq 0,
\end{align}
there exists a Markov process with state space $E$,  with c\`adl\`ag  sample paths such that $(P_t)_{t \geq 0}$ is its transition function, at least if, additionally,
\begin{align}\label{eq:1.2}\tag{1.2}
[0, \infty) \times E \ni (t, x) \longmapsto P_t f(x) \in \mathbb{R} \text{ is continuous for all } f \in C_b(E).
\end{align}
Here $C_b(E),\ \mathcal{B}_b(E)$ denote the set of all real-valued bounded continuous and Borel measurable functions on $E$, respectively.
In this paper we prove that this is not true. More precisely, we present a generic method how to get counter examples easily,  by restricting a c\`adl\`ag Markov process with transition semigroup $(P_t)_{t \geq 0}$, satisfying \eqref{eq:1.1} and \eqref{eq:1.2} to a subset of its state space which is given as the complement of a closed set $N$, which is of measure zero
with respect to its initial distribution and to every kernel $P_t(x, dy),\ x \in E,\ t \geq 0$, but which is hit by the process in finite time with strictly positive probability. In particular, we can take this c\`adl\`ag Markov process to start with as the Brownian motion on $\mathbb{R}^1$ (see Example 2.2 below). 
We present an elementary proof for this in Section 2.1 of this paper (see Theorem 2.1 below).
In Section 2.2,  based on advanced results from potential theory,
we prove something deeper (see \Cref{thm1}),
namely that there exists no Polish topology on the state space $E$ with the same Borel $\sigma$-algebra as the initial one, with respect to which for the semigroup, constructed there, 
there exists a right (not necessarily c\`adl\`ag) Markov process having it as its transition semigroup.

We would like to mention that the above question, without assuming \eqref{eq:1.2},
was originally posed by Y. Le Jan which was answered negatively already in \cite{BeRo11} (see the Example at p. 849) and Brownian motion was identified as a counter example already in \cite[Corollary A.16]{BeCiRo20}. 
One novelty of this paper is to include \eqref{eq:1.2} and give a simpler (and more pedagogical) proof.
The second is to allow (in \Cref{thm1}) a change of the Polish topology.

In Section 2.3 we present an application to quasi-regular Dirichlet forms, giving examples of Dirichlet forms which are not quasi-regular,
by eliminating from the space a set which is not polar.

Finally, in Section 2.4 we give an example of a Hunt process on a locally compact space, for which its transition function is not a Feller semigroup on  the continuous functions vanishing at infinity. 
This completes the well-known  result  on  the construction of a Hunt  process having a given Feller semigroup as transition function, see e.g. Theorem 9.4 from \cite{BlGe68}, showing that its converse does not hold.

\section{Main results}

\subsection{An elementary approach}
Let us introduce the settings and notations through the following three assumptions:
\begin{enumerate}
    \item[$\rm H_1.$] $E$ is a Polish space (i.e. a topological space which is homeomorphic to a complete and separable metric space) endowed with its Borel $\sigma$-algebra $\mathcal{B}(E)$. 
    \item[$\rm H_2.$] $(P_t)_{t\geq 0}$ is a semigroup of Markov kernels on $E$ such that $P_t$ is strong Feller for each $t>0$ and the mapping $[0,\infty)\times E\ni(t,x)\longmapsto P_tf(x)\in \mathbb{R}$ is jointly continuous for each $f\in C_b(E)$.
    \item[$\rm H_\mu.$] There exists a (temporally homogeneous) Markov process $X^\mu:=(X^{\mu}_t)_{t\geq 0}$ on $E$ 
    with a.s. c\`adl\`ag paths, transition function $(P_t)_{t\geq 0}$ and initial distribution $\mu$, that is $X^\mu$ is an $E$-valued (adapted) stochastic process on a filtered probability space $(\Omega, \mathcal{F},\mathcal{F}_t,\mathbb{P}^\mu)$ such that $\mathbb{P}^\mu\circ X_0^{-1}=\mu$ and
    \begin{equation*}
        \mathbb{E}\left[f(X_{t+s}) | \mathcal{F}_t\right]=P_sf(X_t) \quad \mbox{ for all } s,t,\geq 0 \mbox{ and } f\in b\mathcal{B}(E).
    \end{equation*}
    \end{enumerate}
    
    Notice that due to the right continuity 
    %c\'adl\'ag  property 
    of $X^\mu$ and 
    % due 
    to the (strong) Feller property of $(P_t)_{t\geq 0}$, we may (and will) assume without loss that the filtration $(\mathcal{F}_t)_{t\geq 0}$ is right-continuous and complete in $\mathcal{F}$ with respect to $\mathbb{P}^\mu$. 
    In particular, the hitting time of a Borel subset $A\in \mathcal{B}(E)$ given by
    \begin{equation*}
        T_A:=\inf\{t>0 : X_t\in A\}
    \end{equation*}
    is an $\mathcal{F}_t$-stopping time; see e.g. \cite{Ba10} for a short proof that avoids capacity theory.

The main result is the following.
Roughly speaking, it says that if $(P_t)_{t\geq 0}$ admits a c\`adl\`ag Markov process, and if we remove from the state space a set $N$ which is hit by the process, 
but not charged by $P_t,t>0$, then the semigroup restricted to the remaining state space renders a new semigroup which may inherit much of the topological regularity of the unrestricted semigroup, but there is no corresponding c\`adl\`ag Markov process. \

\begin{thm} \label{thm:pathspace}
Let the triplet $(E,(P_t)_{t\geq 0},X^\mu)$ be such that $\rm H_1, \rm H_2, \rm H_\mu$ are  fulfilled, and assume that $N\subset E$ is a closed subset such that 
\begin{equation}\label{eq:N}
    \mu(N)=0,\quad P_t(N,x)=0 \mbox{ for all } t>0, x\in E, \quad \mbox{ and } \quad  
    \mathbb{P}^\mu(T_N<\infty)>0.
\end{equation}
Further, set
\begin{equation}\label{eq:F}
    F:=E\setminus N \quad \mbox{ and } \quad
    Q_t(A,x):=P_t(A,x) \quad \mbox{for all } t\geq 0, x\in F, A\in \mathcal{B}(F),
\end{equation}
where $\mathcal{B}(F)$ is the Borel $\sigma$-algebra on $F$.

Then the pair $(F,(Q_t)_{t\geq 0})$ satisfies $\rm H_1 \mbox{ and } H_2$,  but fails to satisfy $\rm H_{\mu|_{F}}$.
\end{thm}

Before we proceed to the proof of \Cref{thm:pathspace} let us state the following immediate consequence, which reveals that non-existence of a corresponding right-continuous Markov process can even occur even when the jointly continuous semigroup maps bounded functions to $C^\infty$-smooth functions.
\begin{exam}
    Let $(P_t)_{t\geq 0}$ be the 
    Brownian semigroup on $\mathbb{R}$ and set
    \begin{equation*}
    \mathbb{R}_\ast:=\mathbb{R}\setminus \{0\} \quad \mbox{ and } \quad
    Q_t(A,x):=P_t(A,x) \quad \mbox{for all } t\geq 0, x\in \mathbb{R}_\ast, A\in \mathcal{B}(\mathbb{R}_\ast).
\end{equation*}
Then the pair $(\mathbb{R}_\ast,(Q_t)_{t\geq 0})$ satisfies $\rm H_1 \mbox{ and } H_2$, moreover $Q_t(\mathcal{B}_b(\mathbb{R}_\ast))\subset C^\infty_b(\mathbb{R}_\ast)$, but fails to satisfy $\rm H_{\mu}$ for any initial distribution $\mu$ on $\mathbb{R}_\ast$.
This follows directly from \Cref{thm:pathspace} since $(P_t)_{t\geq 0}$ is the transition function of the $\mathbb{R}$-valued Brownian motion which is point recurrent, in particular it hits $\{0\}$ with probability $1$ a.s. for all initial distributions.
\end{exam}

\begin{rem}
We would like to point out that Theorem 2.9 on p. 150 in \cite{vanCast11} contradicts our result from \Cref{thm:pathspace}, 
where the 
%''folklore belief'' 
answer to the open question 
is even strengthened by assuming merely that $(P_t)_{t \geq 0}$ is Feller, i.e. $P_t(C_b(E)) \subset C_b(E)$  for all $t \geq 0$, because all conditions made on $(P_t)_{t \geq 0}$ there (see Definition 2.4 on p. 138/139 of \cite{vanCast11}) are fulfilled e.g. by our counter Example 2.2, when reformulated in the time homogeneous case. So, probably the author just forgot a further condition ensuring the existence of the Markov process claimed in his Theorem 2.9. Such further conditions have been explored extensively and are more or less equivalent to the existence of suitable Lyapunov functions with compact sublevel sets (see \cite{BeRo11}, \cite{BeRo11a},  \cite{BeBo05}, and \cite{BeBoRo06a}). 
\end{rem}

\begin{proof}[\bf Proof of \Cref{thm:pathspace}]
For clarity, let us organize the proof in three steps.

\medskip
\noindent{{\bf Step 1:} $H_1$ holds.} It follows immediately since $F$ is an open subset of the Polish space $E$, hence $F$ is also Polish.

\medskip
\noindent{{\bf Step 2:} $H_2$ holds.} The fact that $(Q_t)_{t\geq 0}$ is a 
semigroup of strong Feller Markov kernels on $F$ is inherited trivially from $(P_t)_{t\geq 0}$ due to the fact that $P_t(E\setminus F,x)\equiv 0$ for all $t>0,x\in E$.
So,  let us show that the mapping $[0,\infty)\times F\ni(t,x)\longmapsto Q_tf(x)\in \mathbb{R}$ is jointly continuous for each $f\in C_b(F)$.
To this end, let $$(0,\infty)\times F\ni(t_n,x_n)\mathop{\longmapsto}\limits_n (t,x)\in [0,\infty)\times F \quad \mbox{and}\quad f\in C_b(F).$$
Let $d$ be a metric that generates the Polish topology on $E$, $\varepsilon>0$ such that $B(x,\varepsilon):=\{y\in E : d(x,y)<\varepsilon\}\subset F$.
Further, let $\widetilde{f}\in \mathcal{B}_b(E)$, $f_0,f_{1} \in C_b(E)$ be such that
\begin{equation*}
    \widetilde{f}(y):=
    \begin{cases}
        f(y), \quad &y\in F\\
        0,    \quad &y\in N=E\setminus F
    \end{cases}
\end{equation*}

\begin{equation*}
    f_0(y):=
    \begin{cases}
        f(y), \quad &y\in B(x,\varepsilon/2)\\
        0,    \quad &y\in E\setminus B(x,\varepsilon)
    \end{cases}
\end{equation*}

\begin{equation*}
    f_{1}(y)=0 \mbox{ for } y\in B(x,\varepsilon/4) \mbox{ and }f_1 \geq |\widetilde{f}-f_0|.
\end{equation*}

%Let us treat first the case when $t>0$, which is further divided in two sub-cases:

%\noindent
%If $t_n\searrow 
Let $t>0$ and $t_0\in (0, t)$.
Then for large enough $n$
then
\begin{equation*}
Q_{t_n}f(x_n)=P_{t_n}\widetilde{f}(x_n)=
P_{t_n-t_0}(P_{t_0}\widetilde{f})(x_n)
\mathop{\longrightarrow}\limits_n P_t\widetilde{f}(x)=Q_tf(x), 
\end{equation*}
where the convergence is ensured by the fact that 
$P_{t_0} \widetilde{f}\in C_b(E)$ and thus 
$(s,x)\longmapsto P_s(P_{t_0} \widetilde{f})(x)$ is jointly continuous.

Let $t=0$. 
Then
\begin{align*}
   |Q_{t_n}f(x_n)-f(x)|=|P_{t_n}\widetilde{f}(x_n)-f_0(x)|\leq|P_{t_n}\widetilde{f}(x_n)-P_{t_n}f_0(x_n)|+|P_{t_n}f_0(x_n)-f_0(x)|.
\end{align*}
By the joint continuity of $(P_t)_{t\geq 0}$ it is sufficient to prove that $\lim\limits_n|P_{t_n}\widetilde{f}(x_n)-P_{t_n}f_0(x_n)|=0$.
But
\begin{align*}
|P_{t_n}\widetilde{f}(x_n)-P_{t_n}f_0(x_n)|\leq P_{t_n}|\widetilde{f}-f_0|(x_n)\leq P_{t_n}f_1(x_n)\mathop{\longrightarrow}\limits_n f_1(x)=0.
\end{align*}

\medskip
\noindent{{\bf Step 3:} $H_{\mu|_F}$ fails.}
Assume that there exists an $F$-valued (temporally homogeneous) Markov process $Y^{\mu|_F}:=(Y^{\mu|_F}_t)_{t\geq 0}$ with a.s. c\`adl\`ag paths in $F$, defined on a filtered probability space $(\Omega', \mathcal{G},\mathcal{G}_t,\mathbb{Q}^{\mu|_F})$ satisfying the usual hypotheses, transition function $(Q_t)_{t\geq 0}$ and initial distribution $\mu|_F$.

Let $D([0,\infty);E)$ be the space of c\`adl\`ag functions from $[0,\infty)$ to $E$ endowed with the (Polish) Skorohod topology. 
Then, we can regard the two processes $X^\mu$ and $Y^{\mu|_F}$ as random variables with values in $D([0,\infty);E)$, whose laws are denoted by $\nu_X:=\mathbb{P}^\mu\circ \left(X^\mu\right)^{-1}$ and $\nu_Y:=\mathbb{Q}^{\mu|_F}\circ \left(Y^{\mu|_F}\right)^{-1}$, respectively.
On the one hand, clearly, these two laws are uniquely determined by their (finite) $n$-dimensional marginals $\nu_X^{t_1,\dots,t_n}$ and $\nu_Y^{t_1,\dots,t_n}$, respectively, for all $n\geq 1$ and $t_1,\dots,t_n\in [0,\infty)$.
On the other hand, by a standard usage of the Markov property for $X^\mu$ and $Y^{\mu|_F}$ we get that for all bounded and measurable functions $f:E^n\longrightarrow \mathbb{R}$
\begin{align*}
    %&\nu_X^{t_1,\dots,t_n}(f)=\int\limits_E P_{t_1}\left(f_1P_{t_2}\left(f_2\cdots P_{t_n}(f_n)\right)\right) \; d\mu\\
    \nu_X^{t_1,\dots,t_n}(f)&=\mathbb{E}^{\mathbb{P}^\mu}\left\{f(X^\mu_{t_1},\dots,X^\mu_{t_n})\right\}\\
    &=\int\limits_E f(x_1,\dots,x_n)P_{t_n}(dx_n,x_{n-1})\dots P_{t_1}(dx_1,x) \; \mu(dx),\\
    \intertext{and since $\mu(F\setminus E)=0= P_t(E\setminus F,x),t>0,x\in E,$ we can continue with}
    &= \int\limits_F f|_F(x_1,\dots,x_n)Q_{t_n}(dx_n,x_{n-1})\dots Q_{t_1}(dx_1,x) \; \mu|_F(dx)\\
    &=\mathbb{E}^{\mathbb{Q}^{\mu|_F}}\left\{f(Y^{\mu|_F}_{t_1},\dots,Y^{\mu|_F}_{t_n})\right\}\\
    &=\nu_Y^{t_1,\dots,t_n}(f).
\end{align*}
Consequently,
\begin{equation}\label{eq:nu}
    \nu_X=\nu_Y.
\end{equation}
Also, note that  by Kuratowski's theorem we have that $D([0,\infty);F)$ is a Borel subset of $D([0,\infty);E)$, and since $\nu_Y$ is supported by $D([0,\infty);F)$, by \eqref{eq:nu} we get
\begin{equation}\label{eq:null}
    \nu_X(D([0,\infty);E)\setminus D([0,\infty);F))=0.
\end{equation}
Moreover, $[T_N<\infty]\subset \left(X^\mu\right)^{-1}\left( D([0,\infty);E)\setminus D([0,\infty);F)\right)$, therefore 
$$
\nu_X(D([0,\infty);E)\setminus D([0,\infty);F))\geq \mathbb{P}^\mu(T_N<\infty)>0,
$$
which contradicts \eqref{eq:null}.
\end{proof}

\subsection{A potential theoretic approach}
% On a question of Le Jan about the Feller semigroups as generators of right Markov processes}

The elementary idea to prove the main result of the previous section was to transfer the problem on the Skorohod space of c\`adl\`ag paths.
There is another approach which we shall adopt in this subsection, based on Ray compactification methods which are developed within the potential theory of Markov processes. 
The result to be obtained in this subsection is in the same spirit as \Cref{thm:pathspace}, however a deeper conclusion comes out. 
To explain this in few words before we dive into more details, recall that in the previous subsection the space $E$ was a priori endowed with a given (Polish) topology, and \Cref{thm:pathspace} provides a way to construct strong Feller semigroups that admit no c\`adl\`ag Markov processes with respect to the given Polish topology.
What we are going to show in the sequel by the potential theoretic approach is that there exists no Polish (or more general, Lusin) topology on the underlying space with respect to which the constructed semigroup admits a right  (not necessarily c\`adl\`ag) Markov process.

Before we proceed, we refer the reader to check the Appendix from \cite{BeCiRo20} for a quick introduction to the basic concepts from probabilistic potential theory, like {\it resolvent}, {\it right process}, {\it excessive} functions and measures, as well as the {\it saturation} of a resolvent, which we shall consider here without further definitions; the above referred Appendix would be sufficient for a non-expert reader to handle the arguments presented in the sequel.

Let $(E,\mathcal{B})$ be a Lusin measurable space and $\mathcal{U}=(U_q)_{q> 0}$ be a sub-Markovian resolvent of kernels on $(E, \mathcal{B})$
and assume that 
 $X=(\Omega, \mathcal{F}, \mathcal{F}_t , X_t, \theta_t , \mathbb{P}^x)$ is a (conservative, for simplicity) right Markov process with state space $E$ having 
 $\mathcal{U}=(U_q)_{q> 0}$ as associated resolvent, that is,
 for all $f\in b\mathcal{B}(E)$, $x \in E$,  and $q >0$ we have
$$
U_q f(x)=\mathbb{E}^{x} \int_0^{\infty}\!\!  e^{-q t} f(X_t) dt;
$$
notice that $E$ is not endowed with a specific topology, whilst the definition of a right Markov process (see \cite[Appendix]{BeCiRo20}) requires that there exists some Lusin topological space with Borel $\sigma$-algebra precisely $\mathcal{B}$, with respect to which the process has right-continuous paths a.s. 
It is then a consequence that a right Markov process has a.s. right-continuous paths in any Lusin topological space on $E$ with Borel $\sigma$-algebra precisely $\mathcal{B}$; see \cite[Corollary 3.10]{BeCiRo20}.

Before we proceed, let us take a moment and explain that when a
{\it generalized Feller property} (e.g. in the spirit of \cite{BeRo11}, page 846) holds for a resolvent $\mathcal{U}$, then the existence of an associated Markov process with a.s. right-continuous paths with respect to some proper Lusin topology on $E$ will automatically imply the existence of a right (hence more regular) Markov process on $E$ with the same resolvent.
A similar result has been presented in \cite{BeCiRo22}, under the aditional assumption that the paths are c\`adl\`ag.

\begin{prop} \label{prop2.1}
Assume that $\mathcal{B}$ is the Borel $\sigma$-algebra of a given Lusin topology $\tau$ on $E$, and $\mathcal{U}$ is a Markovian resolvent of kernels on $E$ with the property that there exists a vector lattice $\mathcal{C}\subset C_b(E)$ such that
\begin{enumerate}[(i)]
\item $1 \in \mathcal{C}$ and there exists a countable subset in $\mathcal{C}$ which separates the points of $E$.
\item We have ${U}_\beta f \in \mathcal{C}$  for all $f \in \mathcal{C}$ and $\beta >0$.
\end{enumerate}

If for all $x\in E$ there exists an a.s. right-continuous temporally homogeneous Markov process $(X_t^x)_{t\geq 0}$ defined on $(\Omega, \mathcal{F},\mathcal{F}_t,\mathbb{P}^{x})$, with resolvent $\mathcal{U}$ and initial distribution $\delta_x$, then there exists a right process on $X'$ on $E$ which has a.s. right-continuous paths with respect to $\tau$, sharing the same resolvent $\mathcal{U}$.
\end{prop}
\begin{proof}
    By dominated convergence and the right-continuity of $X$ it follows that $\lim\limits_{\alpha \to \infty}\alpha {U}_\alpha f=f$ point-wise on $E$ for all $f \in \mathcal{C}$.
Hence we can apply \cite[Proposition 2.1]{BeRo11} and e.g. \cite[Theorem 6.15]{BeCiRo22} to construct a right Markov process $X'=(\Omega', \mathcal{F}', \mathcal{F}'_t, X'_t, \theta'_t, \mathbb{P'}^x)$ on a larger Lusin topological space $(E', \tau')$ such that $E\subset E'$ is $\mathcal{B}(E')$-measurable and $\tau'|_E \subset \tau$, whose resolvent denoted by $\mathcal{U}'$ is an extension of $\mathcal{U}$ from $E$ to $E'$, that is $(U_\alpha'f)|_E=U_\alpha(f|_E)$ for all $\alpha>0$ and $f\in \mathcal{B}_b(E')$.

Let $D$ be the countable set of dyadics in $[0,\infty)$, and denote by $W^E$ (resp.$W^{E'}$) the space of the restrictions to $D$ of all right-continuous paths from $[0,\infty)$ to $E$ (resp. $E'$).
Also, consider the product ${E'}^D$ endowed with the canonical $\sigma$-algebra, and consider the laws of the two processes $X^x$ and $X'^x$ on ${E'}^D$
$$
\nu_{x}:=\mathbb{P}^x\circ (X^x)^{-1} \quad \mbox{and} \quad \nu'_{x}:=\mathbb{P'}^x\circ (X')^{-1} \quad \mbox{ for all } x\in E.
$$
Now, on the one hand any distribution on ${E'}^D$ is uniquely determined by its finite dimensional marginals, hence $\nu_x$ and $\nu_x'$, by the Markov property, are uniquely determined by their one-dimensional marginals. 
The latter marginals are in turn uniquely determined by their Laplace transforms, more precisely by $U_\alpha(f|_{E},x)$ and $U'_\alpha(f,x)$ respectively, for all bounded and $\tau'$-continuous functions $f$ on $E'$.
Since for $x\in E$ we have $U_\alpha(f|_E)(x)=U'_\alpha(f)(x)$ we deduce that as laws on ${E'}^D$,
\begin{equation*}
    \nu_x=\nu_x' \quad \mbox{ for all } x\in E.
\end{equation*}
Next, by \cite{DeMe78}, Chapter IV, pages 91-92 we have that $W^E$ is a universally measurable subset of ${E}^D$, hence of ${E'}^D$. 
Consequently, we can rigorously write 
$$
\nu_x'(W^E)=\nu_x(W^E)=1,
$$
so, under $\mathbb{P'}^x, x\in E$, the paths of $X'$ are restrictions to $D$ of right-continuous paths in $E$ w.r.t. $\tau$. 
Since $\tau'|_E\subset \tau$, we deduce that the entire paths of $X'$ lie in $E$ and are right-continuous with respect to $\tau$. 
This means that the right process $X'$ can be restricted to $E$ and has right-continuous paths there.
\end{proof}

Consider a set $N\in \mathcal{B}$ which is {\it $\mathcal{U}$-negligible}, that is, $U_q(1_N)\equiv 0$ for one (and therefore for all) $q>0$. 
We put $F:= E\setminus N$ and for each $q>0$ we consider the {\it restriction} $U'_q$ of $U_q$ from $E$ to $F$,
$U'_q g:= U_q \overline{g}|_F$ for all $g\in p\mathcal{B}|_F$,  where $\overline{g} $ is any Borel extension of $g$ from $F$ to $E$, i.e., 
$\overline{g} \in p\mathcal{B}$ and $\overline{g}|_F= g$.
Clearly, the family of kernels  $\mathcal{U}': =(U'_q)_{q> 0}$ is a sub-Markovian resolvent of kernels on $(F, \mathcal{B}|_F)$, called 
the {\it restriction} of $\mathcal{U}=(U_q)_{q> 0}$ from $E$ to $F$.

\medskip
The announced main result of this subsection is the following.

\begin{thm}  \label{thm1}
Assume that  
$N\in \mathcal{B}$  is a $\mathcal{U}$-negligible  set which is not polar and consider $\mathcal{U}'=(U'_q)_{q> 0}$, the restriction of 
 $\mathcal{U}=(U_q)_{q> 0}$ from $E$ to $F=E\setminus N$.
 Then there exists no right Markov process on $F$ with resolvent $\mathcal{U}'=(U'_q)_{q> 0}$.
\end{thm}

\begin{proof}
We argue as in \cite{BeRo11}, the Example at page 849 and the proof of \cite[Corollary A.16]{BeCiRo20}.
 Denoting by $Exc(\mathcal{U})$ the space of all $\sigma$-finite excessive measures for $\mathcal{U}_1:=(U_{\alpha+1})_{\alpha>0}$ (for details see \cite[Appendix]{BeCiRo20}, \cite{BeBo04}, and 
 \cite{BeBoRo06}), we have the identification $Exc(\mathcal{U}_1) \equiv Exc(\mathcal{U'}_1)$ given by $\xi \in Exc(\mathcal{U}_1) \Leftrightarrow \big (\xi|_F \in Exc(\mathcal{U'}_1) \mbox{ and } \xi(N)=0 \big )$  
 and one can easily see that $\delta_x\circ U_1$ is extremal in $Exc(\mathcal{U}_1)$ (hence in $Exc(\mathcal{U'}_1)$) for all $x\in F$.  
Therefore, if $F_1$ is the saturation of $F$ (w.r.t. $\mathcal{U'}_1$) then we have the embeddings
$F\subset E \subset F_1$ and $\mathcal{U}=\mathcal{U'}|_F$. 
If $\mathcal{U'}$ is the resolvent of a right Markov process on $F$, then by  \cite[Theorem A.15]{BeCiRo20} 
we get that the set $N\subset F_1\setminus F$ is polar w.r.t. $\mathcal{U}$, which is a contradiction.
\end{proof}

\begin{rem}
 Particular situations from Theorem \ref{thm1} have been considered in \cite{BeRo11}, the Example at page 849,  
 and  \cite{BeCiRo20}, Corollary A.16.
\end{rem}

Let $(P_t)_{t\geqslant 0}$ be the {\it transition function}  of the right Markov process $X$,
$$
P_t f(x)= \mathbb{E}^x f(X_t), x\in E, f\in p\mathcal{B}, t\geqslant 0.
$$
We assume that $P_t$ is a kernel on $(E, \mathcal{B})$ for all $t>0$ and recall that
the resolvent of kernels  associated with $(P_t)_{t\geqslant 0}$  is precisely 
$\mathcal{U}=(U_q)_{q> 0}$, that is
$$
U_q f=\int_0^\infty \!\!\!\! e^{-q t} P_t f d t \ \mbox{ for all }   f\in p\mathcal{B} \mbox{ and }  q>0.
$$

Let  $N\in \mathcal{B}$ be a set which is {\it $(P_t)$-negligible}, that is $P_t (1_N)\equiv 0$  for all $t>0$. 

\begin{rem}
\begin{enumerate}[i)]
    \item It  is clear that a $(P_t)$-negligible set is also $\mathcal{U}$-negligible and observe  that the converse of this statement is not true. Indeed, a simple counterexample is given by the uniform motion to the right on $\mathbb{R}$: we have $P_t f(x)=f(x+t)$ for all all $x\in \mathbb{R}, t\geqslant 0$, and $f\in p\mathcal{B}(\mathbb{R})$. So, if we take $N= \{0\}$ then $P_t(1_N)\not \equiv 0$ for all $t>0$ but  $N$ is  $\mathcal{U}$-negligible.
    \item However, Theorem \ref{thm1} may be applied for the example in assertion $i)$ because the $\mathcal{U}$-negligible set $N= \{0\}$ is not polar for the uniform motion to the right. Consequently, there is no right Markov process on $\mathbb{R}\setminus \{0\}$, having as associated resolvent the restriction to $\mathbb{R}\setminus \{0\}$ of the resolvent of the uniform motion to the right.
\end{enumerate}
\end{rem}

As before, if  $N\in \mathcal{B}$ is a {\it $(P_t)$-negligible} set and $F: =E\setminus N$,  
then we may  consider the restriction  $Q_t$ of each kernel  $P_t $  from $E$ to $F$ and the family 
$(Q_t)_{t\geqslant 0}$ is  a sub-Markovian semigroup of kernels on $(F, \mathcal{B}|_F)$.
In addition, the resolvent of kernels  associated with $ (Q_t)_{t\geqslant 0}$  is precisely 
$\mathcal{U}': =(U'_q)_{q> 0}$.

\begin{coro}
\begin{enumerate}[i)]
    \item Assume that the transition function  $(P_t)_{t\geqslant 0}$ of a right Markov process is a strong Feller semigroup on $E$ endowed with some given Lusin topology whose Borel $\sigma$-algebra is $\mathcal{B}$. Let $N\in \mathcal{B}$ be a {\it $(P_t)$-negligible} set  and $F :=E\setminus N$. Consider the restriction  $ (Q_t)_{t\geqslant 0}$ of $(P_t)_{t\geqslant 0}$ from $E$ to $F$. Then $(Q_t)_{t\geqslant 0}$ is a strong Feller semigroup on the Lusin topological space $F$ with respect to the trace topology, and there exists no right Markov process with state space $F$,  having $(Q_t)_{t\geqslant 0}$ as transition function. 
    \item Suppose in addition that 
 $E$ is Polish space, $N$ is a closed set and the function $[0, \infty)\times E\ni (t,x)\longmapsto P_t f(x)$ is jointly continuous for each $f\in C_b(E)$. Then $(P_t)_{t\geqslant 0}$ is a strong Feller semigroup on the Polish topological space $F$, $P_tf(x)$  is jointly continuous in $(t,x)$ on $[0, \infty)\times F$, and  there is no right Markov process with state space $F$, having $(P_t)_{t\geqslant 0}$ as transition function.
    \item (cf. Corollary A.16. from \cite{BeCiRo20}). Assume that $(P_t)_{t\geqslant 0}$  is  the $1$-dimensional Gaussian semigroup, i.e., the transition function of the real valued Brownian motion. 
    Consider  the restriction $(Q_t)_{t\geqslant 0}$  of $(P_t)_{t\geqslant 0}$ from $\mathbb{R}$ to $\mathbb{R}_\ast:=\mathbb{R}\setminus \{0 \}$. Then $(Q_t)_{t\geqslant 0}$ is a strong Feller semigroup on the Polish topological space $\mathbb{R}_\ast$, $Q_tf(x)$  is jointly continuous in $(t,x)$ on $[0, \infty)\times \mathbb{R}_\ast$, and  there is no right Markov process with state space $\mathbb{R}_\ast$, having $(Q_t)_{t\geqslant 0}$ as transition semigroup. 
    In particular, there exists no right  Markov process associated with $(Q_t)_{t\geq 0}$ no matter what Lusin topology we choose on $\mathbb{R}_\ast$ whose Borel $\sigma$-algebra is the same as the one generated by intervals.
\end{enumerate}
\end{coro}

Let further  $\mathcal{U}=(U_q)_{q> 0}$ 
be a sub-Markovian resolvent of kernels on
the Lusin measurable space $(E, \mathcal{B})$, let $\beta>0$ and denote by
$\mathcal{E}(\mathcal{U}_\beta)$  the set of all $\mathcal{B}$-measurable $\mathcal{U}_\beta$-excessive functions:
$u\in \mathcal{E}(\mathcal{U}_\beta)$ if and only if $u$  is a nonnegative
numerical $\mathcal{B}$-measurable  function, 
$qU_{\beta+q}u \leq  u$ for all $q>0$ and $\lim_{q\mapsto \infty} qU_{\beta+q}u(x)=u(x)$
for all $x\in E$.

Assume that 
$$
\leqno(*)\hspace{10mm}
\sigma({\mathcal
E}({\mathcal U}_\beta))={\mathcal B},
1\in \mathcal{E}(\mathcal{U}_\beta),
\mbox{ and if }
u,v \in {\mathcal E}({\mathcal U}_\beta)
\mbox{ then } \inf (u,v) \in {\mathcal E}({\mathcal U}_\beta).
$$
Notice  that condition $(*)$ does not depend on $\beta> 0$ and
it was proven in \cite{BeRo11}, Corollary 2.3,  that $(*)$   is equivalent with
$$
\leqno(**)\hspace{10mm} \mbox{ for some } \beta>0 
\mbox{ there exists a Ray cone associated with }
{\mathcal U}_\beta.
$$

%for all $x\in E$ we have $\inf (u,v)(x)= \widehat{\inf(u,v)}(x);$
%here  if $w$ is a {\it ${\mathcal U}_\beta$-supermedian function}
% (i.e.,  $\alpha U_{\beta+q} w \leq w$ for all 
%$q >0$), its ${\cal U}_\beta$-excessive regularization
%$\widehat{w}$ is given by $\widehat{w}(x)=\sup_q q U_{\beta+q} w(x).$\\

\noindent
{\bf Polar set.} 
A subset $M$ of $E$ is called ${\it polar}$ provided 
that there exists $u\in \mathcal{E} ({\mathcal U}_\beta)$ such that
$M\subset [u=+\infty]$ and the set
$[u=+\infty]$ is $\mathcal{U}$-negligible.\\

The next corollary shows that in Theorem \ref{thm1}
the hypothesis that  the resolvent $\mathcal{U}=(U_q)_{q> 0}$ is associated with a right process may be replaced by a weaker one, namely, by condition $(*)$ (or equivalently,  by $(**)$).

\begin{coro} \label{cor2.10}
Let  $\mathcal{U}=(U_q)_{q> 0}$ 
be a sub-Markovian resolvent of kernels on
the Lusin measurable space $(E, \mathcal{B})$ such that condition
$(**)$ holds, 
let 
$N\in \mathcal{B}$  be a  $\mathcal{U}$-negligible  set which is not polar,  
and consider $\mathcal{U}'=(U'_q)_{q> 0}$, the restriction of \ 
 $\mathcal{U}=(U_q)_{q> 0}$ from $E$ to $F=E\setminus N$.
 Then there exists no right Markov process on $F$ endowed with the  Lusin topology generated by the Ray cone, with resolvent $\mathcal{U}'=(U'_q)_{q> 0}$.
\end{coro}

\subsection{Resulting examples of non-quasi-regular Dirichlet forms}
Let $(\ce, \cd(\ce))$ be the classical Dirichlet form   associated with the one dimensional Brownian motion.
So, $E=\R$, $\cd(\ce)= H^1(\R)$,
$\ce(u,v)= \frac 1 2 \int_E u'v'd m,$
$m$ being the Lebesgue measure on $\R$,  
and let $\mathcal{U}=(U_q)_{q> 0}$ be its resolvent of kernels.
Let $E_o:= E\setminus\{ 0\}$ and $m_o: =m|_{E_o}$. 

Recall that, according to \cite{AlMaRo93} and \cite{MaRo92} (see also \cite{Fi01}),  the analytic property of a Dirichlet form on an $L^2$-space to be quasi-regular is equivalent with the existence of an associated right Markov process.
We claim that the Dirichlet form $(\ce, \cd(\ce))$ is not quasi-regular, regarded as a  form on  $L^2(E_o, m_o)$.
Indeed, assume that the form $(\ce, \cd(\ce))$ on  $L^2(E_o, m_o)$ is quasi-regular, let  $X^o$ be the associated right Markov process 
with state space $E_o$, and 
$\mathcal{U}^o=(U^o_q)_{q> 0}$ be its resolvent of kernels.
Then $[U_q f|_{E_o} \not= U^o_q(f|_{E_o})]$ is a finely open subset of $E_o$ and it is $m_o$-negligible, therefore it is  $m_o$-polar for every $f\in bp\cb (E)$ and $q>0$. 
Consequently, using a monotone class argument, there exists an $m_o$-inessential set $M\subset E_o$ such that $U^o_q|_{E_o\setminus M}= U_q|_{E_o\setminus M}$, $q>0$, 
and $(U^o_q|_{E_o})_{q>0}$ is the resolvent family of the restriction of $X^o$ to $E_o\setminus N$, a right Markov process
on $E_o\setminus N$. 
On the other hand the set $M:=N\cup \{ 0\}$ is a \ $\cu$-negligible subset of $E$ which is not polar and this leads to a contradiction to
Theorem  \ref{thm1}.
We conclude that the Dirichlet form $(\ce, \cd(\ce))$ is not  quasi-regular, regarded as a form on  $L^2(E_o, m_o)$.

\begin{rem}
$(i)$ We can argue as before for general semi-Dirichlet forms instead of the classical Dirichlet form to get examples of semi-Dirichlet forms which are not quasi-regular, by eliminating a set which is not polar from the space. 

$(ii)$ It was proven in \cite{RoSch95} (see also \cite{BeBoRo06}, the Example on page 279)
that the Dirichlet form associated with the reflecting Brownian
motion is not quasi-regular, considered on $[0, 1)$.
\end{rem}

\subsection{Feller semigroups on continuous functions vanishing at infinity and Hunt processes} 
A classical result on  the construction of Hunt Markov process is the following:
{If $E$ is a locally compact space  with a countable base  then  a sub-Markovian semigroup acting strongly-continuously on the space of continuous functions vanishing at infinity is the transition semigroup of a Hunt process with state space $E$}; cf. Theorem 9.4 from \cite{BlGe68}.
In this frame such a semigroup is called {\rm Feller}. 

We give here a simple example showing that the transition semigroup being a Feller semigroup is not a necessary condition to get a Hunt process.
Indeed, let $b$ be a positive bounded Borelian function on $E$ which is not continuous on $E$ and define
$P_t f := e^{-tb} f$ for all $f \in p\mathcal{B}$.
Then $(P_t)_{t\geq 0}$ is a semigroup of sub-Markovian kernels on $E$,  clearly, it is not a Feller semigroup, however, it is the transition semigroup of a Hunt process
$X=(X_t)_{t\geq 0}$ with state space $E$. 
A rapid argumentation is as follows.
Consider the trivial Markov process $X^ 0=(X^0_t)_{t\geq 0}$ on $E$,  the Markov (actually Hunt) process on $E$ 
for which each point is a trap, i.e., $\Pp^x(X^0_t=x)=1$ for all $t\geqslant 0$ and $x\in E$, or equivalently,
each kernel from its transition semigroup $(Q^0_t)_{t\geqslant 0}$ is the identity operator, $Q^0_t f=f$ for every $t\geqslant 0$ and $f\in p\cb$. 
Furthermore, consider the multiplicative functional induced by $b$,
$M_t= e^{-\int_0^t b(X^0_s) ds}$, $t\geq 0$.
Then take $X$ to be the subprocess of $X^0$, obtained by subordination with the strong multiplicative functional $(M_t)_{t\geq 0}$ and we have 
$P_t f(x)= \mathbb{E}^x(M_t f(X^0_t)), x\in E$.
By Corollary 3.16, page 110 in \cite{BlGe68}, it follows that $X$ is a Hunt process.

\vspace{3mm}

\noindent \textbf{Acknowledgements.} 
Funded by the Deutsche Forschungsgemeinschaft (DFG, German Research Foundation)-SFB 1283/2 2021-317210226. 
For the first and the second authors this work was supported by a grant of the Ministry of Research, Innovation and Digitization, CNCS - UEFISCDI,
project number PN-III-P4-PCE-2021-0921, within PNCDI III.\\

\end{document}